\documentclass[letterpaper, 11pt,  reqno]{amsart}
\usepackage[margin=1.2in,marginparwidth=1.5cm, marginparsep=0.5cm]{geometry}

\usepackage{amsmath,amssymb,amscd,amsthm,amsxtra, esint, xcolor,mathtools,enumerate,hyperref}

\newtheorem{theorem}{Theorem}[section]
\newtheorem{lemma}[theorem]{Lemma}
\newtheorem{proposition}[theorem]{Proposition}

\theoremstyle{definition}
\newtheorem{definition}[theorem]{Definition}

\theoremstyle{remark}
\newtheorem{remark}[theorem]{Remark}

\numberwithin{equation}{section}

\newcommand{\R}{\mathbb{R}}

\newcommand{\eps}{\varepsilon}

\newcommand{\g}{\gamma}

\newcommand{\s}{\sigma}

\newcommand{\wt}{\widetilde}

\newcommand{\dx}{\partial_x}
\newcommand{\dt}{\partial_t}

\newcommand{\LRA}{\Longrightarrow}

\newcommand{\noi}{\noindent}

\newcommand{\N}{\mathbb{N}}

\newcommand{\bul}{\bullet}

\newcommand{\E}{\mathbb{E}}
\newcommand{\ind}{\mathbf 1}
\newcommand{\F}{\mathcal{F}} 
\newcommand{\w}{\textup{w}}
\newcommand{\PP}{\mathbb{P}}
\newcommand{\dl}{\delta}
\newcommand{\al}{\alpha}
\newcommand{\K}{\mathbf{K}}
\renewcommand{\o}{\omega}
\renewcommand{\O}{\Omega}
\newcommand{\les}{\lesssim}

\newcommand{\II}{\text{I \hspace{-2.8mm} I} }

\def\bX{\mathbb{X}}
\def\bY{\mathbb{Y}}

\begin{document}

\title[ASCLT for  HAM with L\'evy noise]{Almost sure central  limit theorem for the hyperbolic Anderson model with L\'evy white noise}

\author{Raluca M. Balan}
\address{Department of Mathematics and Statistics,
University of Ottawa,
150 Louis Pasteur Private,
Ottawa, ON, K1N 6N5,
Canada}

\email{Raluca.Balan@uottawa.ca}


\author{Panqiu Xia}
\address{Department of Mathematics and
Statistics,
Auburn University,
221 Parker Hall,
Auburn, Alabama 36849,
USA}
\email{pqxia@auburn.edu}
\thanks{PX is partially supported by NSF grant DMS-2246850.}

\author{Guangqu Zheng}
\address{ Department of Mathematical Sciences,
University of Liverpool,
Mathematical Sciences Building,
Liverpool, L69 7ZL,
United Kingdom}

\email{guangqu.zheng@liverpool.ac.uk}





\begin{abstract}
In this paper, we present an almost sure central limit theorem
(ASCLT) for the hyperbolic Anderson model (HAM) with a L\'evy white noise
in a finite-variance setting, complementing a recent work   by
Balan and Zheng  ({\it Trans.~Amer.~Math.~Soc.}, 2024) 
 on the (quantitative) central limit theorems for
the solution to the HAM. We provide two different proofs: one uses the
Clark-Ocone formula and takes advantage of the martingale structure
of the white-in-time noise, while the other is obtained by combining
the second-order Gaussian Poincar\'e inequality with
Ibragimov and Lifshits' method of characteristic functions.  
Both approaches are different from
the one developed in the PhD thesis of C. Zheng (2011), allowing us to establish the ASCLT 
without  lengthy computations of star contractions.
Moreover, the second approach is expected to be useful for similar studies
on SPDEs with colored-in-time noises, whereas the former, based on
It\^o calculus, is not applicable.
\end{abstract}

\maketitle


.

\section{Introduction}

 This short note is devoted to the study of   a stochastic (linear) wave equation 
 \noi
\begin{align}
\begin{cases}
&( \dt^2 - \dx^2) u   = u \dot{L} \\[0.5em]
&\big( u(0, \bul), \dt u(0, \bullet) \big) = (1, 0)
\end{cases}
\quad {\rm on}
\quad (t,x) \in\R_+\times\R,
\label{SWE}
\end{align}
 driven by a space-time 
 {\it pure-jump
 L\'evy white noise}; see Section \ref{Sec_levy} for a precise description of the noise and the solution. 


\subsection{Central limit theorems for SPDEs}  \label{SS11}

In a recent work \cite{HNV}, Huang, Nualart, and  Viitasaari established
a central limit theorem (CLT) for the spatial integral of the solution to
a stochastic nonlinear heat equation with space-time Gaussian white noise
and constant initial condition
on $\R_+\times\R$:
\begin{align}\label{she}
( \dt - \tfrac{1}{2} \dx^2) u   =\s(u) \dot{W} 
\quad \text{and} \quad u(0, x) = 1,
\end{align}
where $\dot{W}$ denotes the space-time Gaussian white noise
on $\R_+\times\R$ and the nonlinearity $\s(u)$ is described by
a Lipschitz function $\s: \R\to\R$.  
More precisely,   letting $u = \{u(t,x)\}_{(t,x) \in \R_+ \times \R}$ be the solution to \eqref{she}, and denoting
(with $t_0 > 0$ fixed)
\begin{align}\label{FK}
F_\theta \coloneqq \int_{-\theta}^\theta \big[ u(t_0, x) - 1 \big] dx
\quad
\text{and}
\quad  \s_\theta \coloneqq \sqrt{  {\rm Var}(F_\theta)  },
\end{align}
for any $\theta>0$, the authors of
\cite{HNV} established that $F_{\theta}/\sigma_{\theta}$ admits Gaussian fluctuation
 as $\theta\to \infty$.
This result was partially motivated by the study of mathematical intermittency
of the random field solution.
In retrospect,
(a)
 it is not hard to expect that the solution $u$
is spatially stationary  due to the constant initial condition and homogeneous noise;
(b)
  it is also natural to investigate furthermore the spatial ergodicity of $u$ that will lead
to the first-order result, a law of large number:  as $\theta\to+\infty$,
$
\frac{1}{2\theta} \int_{-\theta}^{\theta}  u(t,x) dx \to 1
$
in $L^2(\mathbb{P})$  and almost surely;
(c)
therefore, the central limit theorem (the second-order fluctuation)
would naturally come into the picture.

Since the appearance of the work \cite{HNV}, there have been a growing body of literature
on similar CLT results for heat equations with various Gaussian noises;
see, for example,
\cite{HNVZ,NZ20a, NSZ21, CKNP21, CKNP22, CKNPjfa,
NXZ22,NZ22, PU22, LP22}.
Meanwhile, such a program was   carried out by Nualart,
Zheng, and their collaborators
to investigate the stochastic  (nonlinear)  wave equations driven by Gaussian noises;
see  \cite{DNZ, BNZ, NZ22, NZ20b, BNQSZ}  and see also the recent works 
\cite{Ebina23a,Ebina23b}
by Ebina on three and higher-dimensional stochastic wave equations. In a recent paper \cite{BZ23}
 by  Balan and Zheng,
 a similar program for the SPDEs with L\'evy noises
 has been initiated; see   Section \ref{Sec_levy} for an overview.

\subsection{Almost sure central limit theorems} \label{SS12}

The almost sure central limit theorem (ASCLT) was first formulated by
Paul L\'evy in his book  \cite[page 270]{Levy54} without a proof.
It had not gained much attention until being rediscovered by various authors 
 in the 1980's
(\cite{Fisher87, Bro88, Schatte88, LP90}). 
The simplest form of ASCLT can be stated as follows.
Let $\{X_n\}_{n\geq 1}$ be an independent and identically distributed sequence
of centered random variables with the unit variance. Then,
for $\PP$-almost every $\omega\in\Omega$, 
\begin{align}\label{AL1}
\frac{1}{\log n} \sum_{k=1}^n \frac1k \dl_{\frac{S_k}{\sqrt{k}}(\omega)} \LRA \g
\quad
\text{as $n\to+\infty$},
\end{align}
 where $\delta_x$ denotes the Dirac mass at $x$, $\gamma$ stands  for the standard normal distribution, 
 and  ``$\LRA$'' represents the weak convergence of finite measures.
In other words, the Gaussian asymptotic behavior can be observed
along a generic trajectory via this logarithmic average.
It was explored in several papers \cite{LP90, BD94, AW96, Les00}
that the above hypothesis of i.i.d. sequence of random variables with mean zero and unit variance is optimal for ASCLT.
See also \cite{BC01, Jon07} for more details. Let us first state
a few definitions, one of which generalizes \eqref{AL1}.

\begin{definition}\label{ASCLT}
{\rm (i) [Discrete version]}
A sequence $\{ F_k\}_{k\in\N}$ of real random variables is said to satisfy
the  \textup{ASCLT} if for $\PP$-almost every
$\omega\in\Omega$,

\noi
\begin{align} \label{def1}
\mu_N^\o \coloneqq
 \frac{1}{\log N} \sum_{k=1}^N \frac{1}{k}
\dl_{F_k(\omega)} \LRA \g
\quad\text{as $N\to\infty$.  }
 \end{align}

 \smallskip
 \noi
 {\rm (ii) [Continuum version]}
 A family  $\{ F_y\}_{y\geq 1}$ of real random variables is said to satisfy
the  \textup{ASCLT} if for $\PP$-almost every
 $\omega\in\Omega$, the map $y \mapsto F_y(\omega)$ is almost surely measurable and
 \footnote{The (random) measure  $\nu_T^\omega$ can be understood by first testing it
 with   $f\in C_c(\R)$ (continuous function with compact support), i.e.,
 ($\star$)\,\,\,
$
 \langle \nu_T^\o, f \rangle =  \frac{1}{\log T} \int_1^T f\big( F_y(\omega) \big) \frac{dy}{y}, 
$
 which defines a positive linear functional on $C_c(\R)$,
 and hence, as a result of Riesz's representation theorem,
 $\nu_T^\omega$ defined in \eqref{def2}
 makes sense as the unique Radon measure satisfying ($\star$)

 }

 \noi
 \begin{align} \label{def2}
\nu_T^\o
\coloneqq \frac{1}{\log T} \int_1^T \dl_{F_y(\omega)}   \frac{dy}{y}
 \LRA \g
 \quad\text{as $T\to+\infty$.}
  \end{align}

 \end{definition}

 In this note, we aim
 to establish an ASCLT for the spatial integrals
 of the solution to an SPDE. Then, it is more natural
 for us to proceed with the continuum version of Definition \ref{ASCLT},
 while the discrete analogue can be dealt with in the same way.

 \begin{remark} \rm   \label{rem1}

 (i) For simplicity,  we take the logarithmic average in \eqref{def1} and \eqref{def2}. 
 In fact, one can consider more general weights, for example, using
 $(d_k, D_N)$ in place of $(\frac{1}{k}, \log N)$,
where
$
d_k > 0$,  $D_N = d_1+ \dots + d_N$ with
$D_{N+1} / D_N\to1$ as $N\to+\infty$; see also the discussion in \cite[page 204]{LP90}.

 \smallskip

 \noi
 (ii)    By the separability of $\R$, there exists a \emph{countable} family $\{ \phi_n\}_{n\in\N}$
 of bounded Lipschitz functions on $\R$ such that $\{ \phi_n\}_{n\in\N}$  forms a
 {\it separating class} for the weak convergent probability measures on $\R$;
 cf.  \cite[Proposition 2.2]{Jon07}.
Thus, the validity of \eqref{def2} (for $\PP$-almost every 
 $\o\in\O$) is equivalent to each of the following statements:\footnote{(b) is 
equivalent to the seemingly stronger statement:
  (b')
 almost surely, for any bounded continuous $f$,
$\frac{1}{\log T} \int_1^T  \frac{1}{\theta} f(F_\theta) \,d\theta
\xrightarrow{T\to+\infty} \int_\R f(x) \g(dx).$}

\noi
\begin{itemize}
\item[(a)] $d_{\rm FM}( \nu^\o_T, \g) \xrightarrow[T \to +\infty]{\rm almost\,\, surely} 0,$
   with $d_{\rm FM}$ defined as in \eqref{FM};

\item[(b)] $\forall f\in C_b(\R)$, almost surely,
$\frac{1}{\log T} \int_1^T  \frac{1}{\theta} f(F_\theta) \,d\theta
\xrightarrow{T\to+\infty} \int_\R f(x) \g(dx);$
here $C_b(\R)$ is the set of real bounded continuous functions on $\R$.
In view of \cite[Proposition C.3.2]{blue}, the above statements are also equivalent to
 the following statements;
\item[(c)] almost surely, $\forall  t\in\R$,
$\frac{1}{\log T} \int_1^T \frac{1}{\theta} \ind_{\{ F_\theta \leq t\} } \, d\theta
\xrightarrow{T\to+\infty}  \gamma(  (-\infty, t]  );$

\item[(d)] $d_{\rm Kol}( \nu^\o_T, \g) \xrightarrow[T\to+\infty]{\rm almost\,\, surely} 0,$
    with $d_{\rm Kol}$ defined as in \eqref{KOL}.

\end{itemize}

 \end{remark}

\noi
The usual CLT does not require the random variables to be defined on a
common probability space, while the ASCLT does; and furthermore,
the convergence to normality at a  fast rate does not necessarily imply the ASCLT.
This can be readily confirmed by, for example, setting $F_k \equiv Y\sim\g$ for all $k$,
 leading to a clear violation of the ASCLT \eqref{def1}.
 In other words, a finer understanding of the whole family of random variables
 is necessary to investigate the ASCLT.

\subsection{SPDEs driven with L\'evy noises}\label{Sec_levy}

 Stochastic differential equations driven by L\'evy processes 
have been studied intensively in the literature since 1970, often using semi-martingale techniques. 
There are already several monographs dedicated to this topic (e.g.,
\cite{bichteler02, protter, applebaum09}).
The study of SPDEs driven by L\'evy noise is a relatively new area in stochastic analysis, 
which extends these techniques to problems that incorporate a spatially-dependent component 
for the noise. These equations can be studied using either the variational approach
(developed at length in the mongraph \cite{PZ07}), or the random field approach 
(see, e.g., \cite{BN15,BN16}). One needs to distinguish between the finite-variance case 
(in which many techniques are similar
to the Gaussian case), and the infinite-variance case (see, e.g., \cite{B14, kosmala-riedle22}). 

 The present paper constitues a new contribution to this area, and focuses on the 
hyperbolic Anderson model in dimension 1, driven by a finite-variance L\'evy noise.
This model can be used for describing the evolution of a wave perturbed by random
forces, which are characterized by a sequence of impulses in space-time 
(described rigorously by a Poisson random measure). Models with this type of
noise could be useful in a variety of situations, when the Gaussian space-time white noise is not 
a correct model for describing the sources of randomness perturbing the system.

Consider
the equation \eqref{SWE} driven
by a space-time L\'evy white noise $\dot{L}$. Throughout this paper, we make the 
following assumptions:


\noi
(i) $\dot{L}$ is the space-time (pure-jump) L\'evy noise on $\R_+\times\R$
with finite variance:
\begin{enumerate}
\item[(i-a)] let $\mathbf{Z} \coloneqq \R_+\times\R\times\R_0$, with $\R_0 \coloneqq \R\setminus\{0\}$ equipped with
the distance $d(x,y)= |x^{-1} - y^{-1}|$, and let $\mathcal{Z}$ be the Borel $\s$-algebra on $\mathbf{Z}$,
and let $m = {\rm Leb}\times   \nu$,
with $ {\rm Leb}$ the Lebesgue measure on $\R_+\times\R$ and
 $\nu$ a $\s$-finite measure on $\R_0$ satisfying
the {\it L\'evy-measure  condition}  $\int_{\R_0} \min\{1, |z|^2\} \nu(dz) <\infty$
and the {\it finite-variance condition} $m_2\coloneqq \int_{\R_0}|z|^2\nu(dz) <\infty$;

\item[(i-b)] let $N$ be a Poisson random measure on the space $(\mathbf{Z}, \mathcal{Z})$ with intensity $m$,
and let $\widehat{N} = N - m$ be the compensated version of $N$, and we put
$
L(A) = \int_{A\times \R_0} z \widehat{N}(ds, dy, dz)
$
for any Borel set $A\subset \R_+\times\R_0$ with   ${\rm Leb}(A) <\infty$,
where the above integral $L(A)$ is defined in the   It\^o sense, which is an {\it infinitely divisible}
random variable  with L\'evy-Khintchine formula
$
\E\big[ e^{i\lambda L(A)}\big]
= \exp\big( {\rm Leb}(A)\int_{\R_0}  (e^{i\lambda z}  - 1 - i\lambda z  ) \nu(dz) \big)
$, $\lambda \in\R$;

\item[(i-c)] $\dot{L}(s, y) = L(ds, dy)$ is the formal derivative $\partial_s \partial_y L$;
\end{enumerate}

\smallskip
\noi
(ii) the product of the unknown $u$ and the L\'evy noise $\dot{L}$ is intetpreted
in the It\^o sense;  and in Duhamel formulation, we rewrite \eqref{SWE}
in the following integral form:
\begin{align}\label{mild}
u(t,x) = 1 + \int_0^t \int_\R G_{t-s}(x-y) u(s,y) L(ds, dy),
\end{align} 
and $G_t(x)= \frac{1}{2}\ind_{\{ |x|< t\}}$ denotes the wave kernel.\footnote{The integral 
in \eqref{mild} 
is understood as 
$
\int_{\R\times\R} X(s, y) L(ds, dy) = \int_{\R\times\R\times\R_0} X(s, y)z \widehat{N}(ds, dy, dz)$,
whenever the right side exists.}

It is known that equation \eqref{SWE}  admits
 a unique solution $u = \{u(t,x)\}_{(t, x) \in \R_+ \times \R}$
 that is  an $\mathbb{F}$-predictable process satisfying \eqref{mild}
 and
 $
 \sup_{(t,x) \in [0,T] \times \R} \E [ | u(t,x)|^2  ] <\infty
 $
 for any finite $T > 0$; see \cite[Theorem 1.1]{BN16}.
 Here, $\mathbb{F}=\{ \F_t\}_{ t\in\R_+}$ denotes the natural filtration
 generated by the L\'evy noise $\dot{L}$.\footnote{More precisely,
let $\F^0_t$ be the $\s$-algebra
generated by the random variables
$N([0,s]\times A \times B)$ with $s \in [0,t]$ and ${\rm Leb}([0,s]\times A)+ \nu(B)<\infty$.
And let $\F_t = \s\big( \F_t^0 \cup \mathcal{N} \big)$ be the
$\s$-algebra generated by $\F_t^0$ and  the set $\mathcal{N} $ of $\PP$-null sets.
This gives us a filtration $\mathbb{F}=\{ \F_t\}_{t\in\R_+}$.
 }

In a recent work \cite{BZ23}, Balan and Zheng established  
the following  result on spatial ergodicity and CLT results 
for \eqref{SWE}.

\begin{theorem}[{\cite[Theorem 1.1]{BZ23}}]\label{thm-cls-clt}
Let $u = \{u(t,x)\}_{(t,x) \in \R_+ \times \R}$ be the solution to \eqref{SWE}. Fix $t_0 > 0$, and let $F_{\theta}$ and $\sigma_{\theta}$ be defined as in \eqref{FK} for $\theta > 0$. Then,
\begin{enumerate}[(i)]
\item[\rm(i)] $\{u(t_0, x)\}_{x \in \R}$ is strictly stationary and ergodic.

\item[\rm(ii)] $\displaystyle \sigma_{\theta} = \sqrt{{\rm Var} (F_{\theta})} \asymp \sqrt{\theta}$ as $\theta \to \infty$.

\item[\rm(iii)] Assume additionally that $m_{2 + 2\alpha} + m_{1 + \alpha} < \infty$ for some $\alpha \in (0,1]$,\footnote{In particular, this forces $m_2$ to be finite. \label{note3}} where
\begin{align}\label{def_mp}
m_p\coloneqq \int_{\R_0} |z|^p \nu(dz), \quad p \in[1,\infty).
\end{align}
Then, $
  {\rm dist} \big(\tfrac{F_{\theta}}{\sigma_{\theta}}, \mathcal{N}(0,1)\big) \lesssim \theta^{-\frac{\alpha}{1 + \alpha}},$
where the implicit constant  does not depend on $\theta$, and one can choose
the distributional metric {\rm dist} to be one of the following: Fortet-Mourier,
1-Wasserstein, and Kolmogorov distances {\rm(}see \eqref{FM}-\eqref{KOL}{\rm)}.
\end{enumerate}
\end{theorem}

Note that $\s_\theta > 0$ for $\theta > 0$; see  \cite[Remark 4.3]{BZ23}.
One of the key ingredients for part (iii) is a second-order
Poincar\'e inequality that goes back to  Last, Peccati, and Schulte's paper
\cite{LPS16},
and has been recently improved by Trauthwein \cite{Tra22}.

The goal of the current note is
to provide the following ASCLT.

\begin{theorem} \label{main}
With the notation as  in Theorem \ref{thm-cls-clt} and
let $\wt{F}_\theta =   F_\theta/\s_\theta$ for $\theta > 0$.
Assume $m_{1+\al} + m_{2+2\al} < \infty$
for some $\al\in(0, 1]$.
Then, $\{\wt F_\theta\}_{\theta \geq 1}$
satisfy the {\rm ASCLT}.
\end{theorem}

The paper is organized as follows. In Section \ref{sec_pre}, we collect some preliminaries
 for the proofs of Theorem \ref{main} presented
in Section \ref{S2}. Proofs of technical lemmas will be given in the Appendix \ref{APP}.

\section{Preliminaries}\label{sec_pre}

\subsection{Notations}
 In this paper, all random objects are defined on a rich common probability space
$(\Omega, \mathcal{F},\PP)$, and $\E$, ${\rm Var}, {\rm Cov}$ stand for the associated expectation, variance, covariance
operators respectively. We use  $\|X\|_p \coloneqq \| X\|_{L^p(\Omega)}  $
for any $p\in[1,\infty)$ and real-valued random variable $X$,  and we write
$\| f \|_\infty$ for the essential-sup norm for any measurable $f:\R\to\R$.
 We write $Y\sim\g$ to mean that $Y$ is a standard
 normal random variable.
 For any two (probability) measures $\mu_1$
 and $\mu_2$ on $\R$, we define the Fortet-Mourier distance

 \noi
 \begin{align}\label{FM}
 d_{\rm FM}(\mu_1, \mu_2) \coloneqq \sup_{\|\phi\|_{\infty} + \|\phi\|_{\infty} < 1}  \Big| \int_\R  \phi(x) \mu_1(dx) -\int_\R  \phi(x) \mu_2(dx) \Big| .
 \end{align}
 
 \noi
 It is well-known that $d_{\rm FM}$ characterizes
 the weak convergence of probability measures on $\R$;
 see \cite[Theorem 11.3.3]{Dudley}.
Another two stronger metrics, the Wasserstein distance $d_{\rm Wass}$
and the Kolmogorov distance  $d_{\rm Kol}$, are defined, respectively,
 by
 \begin{align}\label{Wass}
 d_{\rm Wass}(\mu_1, \mu_2)
 &\coloneqq  \sup_{\|\phi'\|_\infty\leq 1}  \Big| \int_\R  \phi(x) \mu_1(dx) -\int_\R  \phi(x) \mu_2(dx) \Big| ,\\
 d_{\rm Kol}(\mu_1, \mu_2)
 &\coloneqq \sup_{t\in\R} |\mu_1(  (-\infty, t]  ) - \mu_2(  (-\infty, t]  ) |.
 \label{KOL}
 \end{align}

\noi
 For two random variables $X$ and $Y$, we also write $d_{\rm Wass}(X, Y)$
 for the Wasserstein distance between their distributions.
On the $L^2$ probability space generated by the Poisson random measure $N$,
one can develop the Malliavin calculus; see \cite[Section 2]{BZ23}
for more details and for any unexplained notation.
In particular, we recall that  for $F\in \mathbb{D}^{1,2}$,
$D_\xi F$
denotes the Malliavin derivative $DF\in L^2(\Omega; L^2(\mathbf{Z}, m))$
valued at $\xi\in \mathbf{Z}$, where for convenience,  $\xi$ and $m(d\xi)$ are short for $\xi = (r, y, z)\in\R_+\times\R\times\R_0$
and $m(d\xi) = dr \times dy \times \nu(dz)$ respectively,
whenever no ambiguity arises. For two functions $g(t)$ and $h(t)$,
$g(t)\asymp h(t)$ means that
$0 < c_1 \leq \liminf_{t\to+\infty} g(t)/h(t) \leq   \limsup_{t\to+\infty} g(t)/h(t) \leq c_2 <\infty$
for some constants $c_1$ and $c_2$.


\subsection{Preliminary results} 

\begin{proposition}[{\cite[Proposition 3.2]{BZ23}}] \label{prop:BZ23} Let $m_p$ be defined as in \eqref{def_mp} for $p \geq 1$.
Assume   $m_2 + m_{2+2\al} < \infty$ for some $\alpha \in (0,1]$.
Then, for any $0 < r_1 < r_2 < t_0$,
\begin{gather}
\| D_{r_1, y, z} u(t_0, x) \|_{2+2\al} \les G_{t_0 -r_1}(x-y) |z|,   \nonumber 
\shortintertext{and}
\begin{aligned}
\| D_{r_2, y_2, z_2}  D_{r_1, y_1, z_1} u(t_0, x) \|_{2+2\al}
\les    |z_1 z_2| G_{t_0 - r_2} (x - y_2) G_{r_2 - r_1} (y_2 - y_1). 
 \end{aligned}
\label{MD2}
\end{gather}


\end{proposition}


The following proposition is a variant of {\cite[Theorem 3.4]{Tra22}} for  random variables
whose  variances are not necessarily one. The modification  is   standard within the Stein's method 
 (\cite{Ross11, CGS, blue}).
 For the completeness, we provide a sketchy proof.

First, we recall the framework used in \cite{Tra22}. 
Let $(\bX,\mathcal{X},\lambda)$ be a $\sigma$-finite measure space, and ${\bf N}_{\bX}$ be the set of $\mathbb{N}_0 \cup\{\infty\}$-valued measures on $\bX$. Let $\mathcal{N}_{\bX}$ be the smallest $\sigma$-field on ${\bf N}_{\bX}$ for which all maps ${\bf N}_{\bX} \ni \xi \mapsto \xi(B)$ are measurable
 for $B \in \mathcal{X}$.

Let $\chi$ be a Poisson random measure (PRM) on $\bX$ of intensity $\lambda$, defined on a probability space $(\Omega,\mathcal{F}, \mathbb{P})$. 
For any $\chi$-measurable random variable $F$, there exists a measurable function $f:\mathbb{N}_{\bX}\to \R$ such that $F=f(\chi)$ a.s. For such a variable, we define the add-one cost operator
$D_{x}F=f(\chi+\delta_x)-f(\chi)$ for all $x \in \bX$. 
It is known that for $F\in\mathbb{D}^{1,2}$, $D F$ coincides with the   add-one cost operator  
and hence we will simply use the same notation; see \cite[Remark 2.7]{BZ23}.
In our framework, $(\bX,\mathcal{X},\lambda)=(\mathbf{Z},\mathcal{Z},m)$.
If $\eta$ is a PRM on $\bY=\bX \times [0,1]$ of intensity $\overline{\lambda}=\lambda \otimes dt$, then $\chi=\eta(\cdot \times [0,1])$ is a PRM on $\bX$ of intensity $\lambda$, and for any $\chi$-measurable random variable $F$, $D_{(x,s)}F=D_x F$ for all $(x,s) \in \bY$.

\begin{proposition} 
\label{prop:tara}
Let $\chi$ be a Poisson random measure on $\bX$ of intensity $\lambda$. 
Let $F \in \mathbb{D}^{1,2}$ be such that $\E(F)=0$ and $\E(F^2)=\sigma^2$. Then,

\noi
\begin{align}\label{tara1}
d_{\rm Wass}(F, \mathcal{N}(0,1)) \leq \frac{2}{\sqrt{\pi}}|1 - \sigma^2 | + \g_1 + \g_2 + \g_3,
\end{align}

\noi
where, for  $p,q\in (1,2]$,

\noi
\begin{align}
\begin{aligned}
\g_1&\coloneqq  \frac{2^{\frac2p + \frac12} }{\sqrt{\pi}}
\bigg( \int_{\bX}  \bigg[ \int_{\bX}  \| D_{x_2}F \|_{2p}  \| D_{x_1}D_{x_2}F\|_{2p}  \,  \lambda(dx_2) \bigg]^p
\lambda(dx_1) \bigg)^{\frac1p} ,
\\
\g_2&\coloneqq\frac{2^{\frac2p - \frac12} }{\sqrt{\pi}}
\bigg( \int_{\bX}  \bigg[ \int_{\bX}   \| D_{x_1}D_{x_2}F\|^2_{2p}  \,  \lambda(d\xi_2) \bigg]^p
\lambda(dx_1) \bigg)^{\frac1p} ,
\\
\g_3&\coloneqq 2    \int_{\bX} \E|D_{x} F|^{q+1}  \, \lambda(dx).
\end{aligned}
\label{tara2}
\end{align}

\end{proposition}

\begin{proof}

Let $\eta$ be given as above, then we can view $F$ as a $\eta$-measurable random variable $F$
with $D_{y}F=D_x F$ for all $y = (x,s) \in \bY$.

We use \cite[(3.3)]{Tra22} (whose proof does not rely on the fact that $\E(F^2)=1$):
\begin{align}\label{taraT1}
d_{\rm Wass}(F, \mathcal{N}(0,1)) 
& \leq  \sqrt{2/\pi}  \E \left|1- \sigma^2 -  \mathbf{G} \right|+2 \int_{\bY} |E[D_y F|\eta] | \cdot |D_y F|^q \overline{\lambda}(dy),
\end{align} 
where
$
\mathbf{G}:=\int_{\bY} D_{y}F \, \E[D_y F|\eta] \overline{\lambda}(dy)-\sigma^2
$
  has mean zero in view of  \cite[(D.20)]{Tra22},
 and the second term in \eqref{taraT1} is bounded by $\gamma_3$ 
 in view of \cite[(D.47)]{Tra22}.   
 The same argument as for \cite[(D.18)]{Tra22} shows that
$\E|\mathbf{G}| \leq \beta_1'+\beta_2'$,
where $\beta_1'$ and $\beta_2'$ have the same expressions as $\beta_1,\beta_2$ in \cite[Theorem 3.2]{Tra22}, {\em but without} $\sigma^{-2}$, and without $\sqrt{2/\pi}$ (which is a typo in \cite{Tra22}). Hence,
the above Wasserstein distance \eqref{taraT1} is bounded by 
$\sqrt{2/\pi}  |1-\sigma^2|+  \sqrt{2/\pi}   \E|\mathbf{G}|+\gamma_3
$.
Then $\sqrt{2/\pi} (\beta_1'+\beta_2')=\beta_1+\beta_2$, where $\beta_1$ and $\beta_2$ are given as in \cite[Theorem 3.2]{Tra22} without the factor $\sigma^{-2}$. The conclusion follows since $\beta_1+\beta_2 \leq \gamma_1+\gamma_2$, by \cite[D.46]{Tra22}.
\end{proof}

\begin{remark}\rm
Proposition \ref{prop:tara} will be combined with Ibragimov-Lifshits' method of characteristic functions (Proposition \ref{prop:IL}) to provide the second proof of Theorem \ref{main} in Section \ref{ss_prf2}.
It is important to note that  $F$ in \eqref{tara1} may not
have the unit variance, whereas the original statement of \cite[Theorem 3.4]{Tra22} gives an estimate 
for $d_{\rm Wass}(\widehat{F},Y)$ with $\widehat{F}=(F-\E (F))/\sqrt{{\rm Var}(F)}$.
This seemingly minor modification is crucial for obtaining \eqref{claima}: applying directly \cite[Theorem 3.4]{Tra22}
would require a uniform (positive) lower bound for the variance $V_{\theta, \w}$ in \eqref{VTW}, 
but  $V_{\theta, \theta} =0$. 
\end{remark}

\section{Two proofs of Theorem \ref{main}} \label{S2}

\subsection{Proof of Theorem \ref{main} via the Clark-Ocone formula}
\label{ss_prf1}

In view of Remark \ref{rem1}-(ii), it suffices to show

\noi
\begin{align}\label{show1}
\frac{1}{\log T} \int_1^T \frac{1}{\theta} f( \wt{F}_\theta) d\theta
\xrightarrow[T\to +\infty]{a.s.}
\int_\R f(x)\g(dx)
\end{align}
for any bounded Lipschitz function $f$ on $\R$ with the Lipschitz constant ${\rm Lip}(f)$.
By Theorem \ref{thm-cls-clt}-(iii),
$\lim_{\theta \to \infty}\E [f(\wt{F}_\theta)] = \int_\R f(x)\g(dx)$.
Then, \eqref{show1} is equivalent to

\noi
\begin{align}\label{show2}
\frac{1}{\log T} \int_1^T \frac{1}{\theta} \big(  f(\wt{F}_\theta) - \E[ f(\wt{F}_\theta) ] \big)  d\theta
\xrightarrow[T\to+\infty]{a.s.}
0.
\end{align}
Put
\begin{align}\label{def:h}
H_\theta\coloneqq f(\wt{F}_\theta) - \E[ f(\wt{F}_\theta)].
\end{align}

\noi
Then, one can deduce from the Clark-Ocone
formula (cf. \cite[Lemma 2.5]{BZ23}) that

\noi
\begin{align}\label{com0}
H_\theta = \int_0^{t_0}\int_\R\int_{\R_0} \E\big[ D_{r,y,z} H_\theta | \mathcal{F}_r \big]
\widehat{N}(dr, dy, dz),
\end{align}

\noi
which is an It\^o integral with respect to the compensated Poisson random measure
$\widehat{N}$.
Then, we deduce from the It\^o isometry, 
the chain rule (e.g. \cite[(2.42)]{BZ23}), the Cauchy-Schwarz, and Jensen's inequalities
   that
 for $\theta < \w$,

 \noi
 \begin{align}
\big|  \E [ H_\theta H_\w  ] \big|
&\leq  \frac{ {\rm Lip}^2(f)}{\s_\theta \s_\w}  \int_0^{t_0} \int_\R\int_{\R_0} \| D_{r, y, z } F_\theta \|_2
\cdot  \| D_{r, y, z } F_\w \|_2 \,  dr  dy \nu(dz);
\label{com1}
 \end{align}
 while Proposition \ref{prop:BZ23}, together with Minkowski's inequality,
 implies that

 \noi
 \begin{align}
  \| D_{r, y, z } F_\theta \|_2
\leq  \int_{-\theta}^\theta  \| D_{r, y, z} u(t_0,  x) \|_2 \, dx\les |z|  \int_{-\theta}^\theta  G_{t_0 -r}(x-y) dx.
  \label{com2}
 \end{align}
 Thus, it follows from \eqref{com1}, \eqref{com2}, Theorem \ref{thm-cls-clt}-(ii), 
 and the following inequalities
 \[
 G_{t_0 - r}(\bul) \leq G_{t_0}(\bul),
 \quad
 G_{t_0}(x_1 - y) G_{t_0}(x_2-y) \leq  G_{t_0}(x_1 - y) G_{2t_0}(x_1 - x_2)
 \]
 that
 \begin{align}
 \begin{aligned}
 \big|  \E [ H_\theta H_\w  ] \big|
& \les
 \frac{1}{\sqrt{ \theta \w} }
 \int_0^{t_0}   \int_\R\int_{\R_0} |z|^2 \left(  \int_{-\theta}^\theta  G_{t_0 -r}(x_1-y) dx_1  \right) \\
 &\qquad\qquad\quad \cdot  \left(  \int_{-\w}^\w  G_{t_0 -r}(x_2-y) dx_2  \right)  drdy  \nu(dz) \\
 &\leq  \frac{m_2 t_0 }{\sqrt{ \theta \w} }
   \int_\R  \left(  \int_{-\theta}^\theta  G_{t_0}(x_1-y) dx_1  \right) \cdot 
    \left(  \int_{-\w}^\w  G_{2t_0}(x_2-x_1) dx_2  \right)  dy \\
   &\leq   4m_2 t^3_0  \big( \theta / \w \big)^{\frac12}
   \qquad\text{for $\theta < \w$,}
 \end{aligned}
 \label{com3}
 \end{align}
 where the last step is obtained by integrating in the order of $dx_2, dy$, then $dx_1$.
 Hence, the proof  of \eqref{show2} is done by invoking the following Lemma \ref{lem1}. \qed


\begin{lemma}\label{lem1}
 Let $\{ H_\theta\}_{\theta > 0}$ be a family uniformly bounded random variables with
 \[
 \big| \E[ H_\theta H_\w ] \big| \leq C_\beta  \big( \theta / \w \big)^{\beta}
 \]
 for any $0 < \theta < \w$, 
 where the exponent $\beta > 0$ and the constant $C_\beta$
 do not depend on $(\theta, \w)$. 
 Assume also that $\theta\mapsto H_\theta$ is a measurable
 function almost surely.
 Then,

 \noi
 \begin{align} \label{com4}
L_T\coloneqq \frac{1}{\log T} \int_1^T \frac{1}{\theta} H_\theta d\theta \xrightarrow[T\to+\infty]{a.s.} 0.
 \end{align}
 \end{lemma}
 A discrete analogue of Lemma \ref{lem1} can be formulated in a straightforward manner. 
 Lemma \ref{lem1} follows essentially from a  Borel-Cantelli argument, 
 and for the sake of completeness, 
 we present a short proof   in the Appendix \ref{APP}.

 \begin{remark}\rm 
The aforementioned
Lemma \ref{lem1} for $H_\theta = f( \wt{F}_\theta) - \E\big[f(\wt{F}_\theta)\big] $ with bounded Lipschitz $f$ is enough to prove Theorem \ref{main}.
A crucial point is the usage of the Clark-Ocone representation $H_\theta$
as an It\^o integral \eqref{com0}, and applying the moment estimate
of the Malliavin derivatives (Proposition \ref{prop:BZ23}).
This idea has been used in papers \cite{LZ1, LZ2, LZ3}
for similar studies on SPDEs driven by Gaussian noises.
However when the noise is not white in time (cf. \cite{BNQSZ, NXZ22}),
the strategy, based on It\^o's calculus, is not applicable any more. This motivates us to provide another proof using a combination of the Ibragimov-Lifshits' criteria for ASCLT   and the second-order Gaussian Poincar\'e inequalities; 
see also Remark \ref{rem3} for a further discussion.

 \end{remark}

\subsection{Proof of Theorem \ref{main} based on the Ibragimov-Lifshits' criterion}\label{ss_prf2}

%


Let us first state a variant of the Ibragimov-Lifshits' result.

 \noi
 \begin{proposition}\label{prop:IL} {\rm(}\cite[Ibragimov-Lifshits' criterion]{IL99}{\rm)}
A family of random variables $\{F_\theta\}_{\theta\geq 1}$ satisfies the {\rm ASCLT}
if $\theta\mapsto F_\theta$ is measurable almost surely, and

 \noi
\begin{align}
    \sup_{|s| \leq T}
    \int_2^\infty
    \dfrac{\E \big[  |\K_t(s)|^2\big]  }{t  \log t} dt <\infty,
  \label{IL}
\end{align}
for any finite $T >0$, where
\begin{align}\label{Kn}
\K_t(s)
\coloneqq \frac{1}{\log t} \int_1^t \frac{1}{\theta}
\big( e^{is F_\theta} - e^{-s^2/2} \big)  d\theta,
\quad t \in( 1, \infty).
\end{align}
\end{proposition}

In Ibragimov-Lifshits' original paper \cite{IL99},
the criterion
is proved for the discrete version (see Definition \ref{ASCLT}-(i)).
For readers' convenience, we provide a sketchy proof of Proposition \ref{prop:IL}
 in the Appendix \ref{APP} though it is almost identical to that in \cite{IL99}.


Let $F_\theta$ be defined as in \eqref{FK}
and  $\K_t(s)$ be given by \eqref{Kn} with $F_\theta$ replaced
by $\wt{F}_\theta = \frac{1}{\s_\theta} F_\theta$.
By expanding $|\K_t(s)|^2$, we   write
\begin{align*}
 |\K_t(s)|^2
 &=  \frac{1}{(\log t)^2}
 \int_{[1,t]^2} \frac{1}{\theta\w}
\big( e^{is \wt F_\theta} - e^{- \frac{s^2}{2} } \big)\big( e^{- is \wt F_\w} - e^{- \frac{s^2}{2} } \big) d\theta d\w  \\
 &=  \frac{1}{(\log t)^2}
 \int_{[1,t]^2} \frac{1}{\theta\w}
       \big( e^{is (\wt F_\theta - \wt F_\w)} + e^{-s^2} - e^{is \wt F_\theta } e^{- \frac{s^2}{2} }
          - e^{ - is \wt F_\w } e^{-\frac{s^2}2 } \big) d\theta d\w  \\
 &=  \mathbb{I}_t(s) -  e^{ - \frac{s^2}{2}}   \II_t(s),
 \end{align*}

 \noi
 where

 \noi
 \begin{align}\label{12s}
 \begin{aligned}
 \mathbb{I}_t(s)
 &\coloneqq
  \frac{1}{(\log t)^2}
 \int_{[1,t]^2} \frac{1}{\theta\w}\Big( e^{i s (\wt F_\theta - \wt F_\w)} -  e^{-s^2} \Big)\, d\theta d\w ,
\\
  \II_t(s) &\coloneqq  \frac{1}{\log t} \int_1^t \frac{1}{\theta} \big( e^{i s \wt F_\theta } 
  + e^{- i s \wt F_\theta } - 2 e^{-\frac{s^2}{2}} \big) \, d\theta .
 \end{aligned}
 \end{align}

 \noi
Therefore, thanks to Proposition \ref{prop:IL}, it suffices to show that
\begin{align}
\notag 
A_1(s) \coloneqq \int_2^\infty \frac{  \E\big[ \mathbb{I}_t(s) \big] }{t\log t} dt
 \quad
  \text{and}
   \quad
    A_2(s) \coloneqq \int_2^\infty \frac{  \E\big[ \II_t(s) \big] }{t\log t} dt,
    \quad s\in [-T, T]
\end{align}

\noi
are both uniformly bounded for any given $T>0$.


$\star$ {\it Estimation of the  $A_2$ term}.
Applying Euler's formula
$
e^{i z} + e^{- i z} = 2 \cos(z)$, $z\in\mathbb{R}$,
and
  $\E [e^{i s Y}] = e^{- \frac{s^2}{2}}$ for $Y\sim \g$, one can write,
with  $\phi_s(z) = 2\cos(s z)$, that

\noindent
\begin{align}
\begin{aligned}
\big| \mathbb{E}\big( e^{i s \wt F_\theta } + e^{- i s \wt  F_\theta} - 2 e^{-\frac{s^2}{2}} \big)   \big|
&=  \big| \mathbb{E}\big( \phi_s(\wt F_\theta)  -  \phi_s(Y)  \big)   \big|
  \leq 2 |s|  d_{\rm Wass}(\wt F_\theta, Y) ,
\end{aligned}
\label{T4}
\end{align}

\noindent
where we used 
\eqref{Wass}
and the fact that
$\phi_s$ is a Lipschitz function with ${\rm Lip}(\phi_s) \leq 2|s| \leq 2T$.
Then, it follows from \eqref{12s}, \eqref{T4}, and Theorem \ref{thm-cls-clt}-(iii) that
\begin{align*}
\sup_{|s|\leq T}\big|   \E [ \II_t(s)   ] \big|
&\les_T \frac{1}{\log t} \int_1^\infty \frac{d\theta}{\theta^{1 + \frac{\al}{1+\al}}}
\les_T \frac{1}{\log t},
\end{align*}

\noi
where $\lesssim_T$ suggests the implicit constant in the inequality only depends on $T$.
   Therefore, it holds that $\sup\{ |A_2(s)| : s\in [-T, T]   \} <\infty$ for any finite $T>0$.


$\star$
 {\it Estimation of the $A_1$ term.} We first write

 \noi
 \begin{align}
 \sup_{|s|\leq T} A_1(s)
& \leq 2\sqrt{2} T  \int_2^\infty \frac{1 }{t (\log t)^3}
 \bigg( \int_{[1,t]^2} \frac{1}{\theta\w} \,
d_{\rm Wass}\big( \tfrac{\wt F_\theta - \wt F_\w}{\sqrt{2}}, Y \big)
 \, d\theta d\w   \bigg)  dt  \notag \\
 &=4\sqrt{2} T  \int_2^\infty \frac{1 }{t (\log t)^3}
 \bigg( \int_{1 < \theta < \w < t} \frac{1}{\theta\w} \,
 d_{\rm Wass}\big( \tfrac{\wt F_\theta - \wt F_\w}{\sqrt{2}}, Y \big)
 \, d\theta d\w  \bigg)  dt ,
  \label{T5a}
 \end{align}

 \noi
 where we used $d_{\rm Wass}(-X, Y) =d_{\rm Wass}(X, Y)$ for $Y\sim \g$
 and the following bound
\[
\big| \E\big[ e^{i s (\wt F_\theta - \wt F_\w)} -  e^{-s^2}  \big] \big|
= \big| \E\big[ e^{i \sqrt{2} s (\frac{\wt F_\theta - \wt F_\w}{\sqrt{2}})} - e^{i \sqrt{2}s Y}  \big] \big|
\leq 2\sqrt{2}|s|  d_{\rm Wass}\big( \tfrac{\wt F_\theta - \wt F_\w}{\sqrt{2}}, Y \big).
\]


We {\bf claim} that there exist positive real numbers $\beta_1, \beta_2$, and $\beta_3$
such that

\noi
\begin{align}\label{claima}
d_{\rm Wass}\big( \tfrac{\wt F_\theta - \wt F_\w}{\sqrt{2}}, Y \big)
\les  \theta^{-\beta_1}  + \w^{-\beta_2}  + (\theta/\w)^{\beta_3}
\end{align}

\noi
for   $1 < \theta < \w <\infty$. It is then easy to deduce the finiteness of \eqref{T5a}
from the  claim \eqref{claima}. The rest of the proof is then devoted to
verifying the claim \eqref{claima}.


Observe that $\wt F_\theta$ and  $\wt F_\w$ are centered with unit variance, and thus,
\begin{align}
V_{\theta,\w}\coloneqq \E\Big[  \big( \tfrac{\wt F_\theta - \wt F_\w}{\sqrt{2}} \big)^2 \Big]
= 1 - {\rm Cov}(   \wt F_\theta , \wt F_\w).
\label{VTW}
\end{align}
It follows from Proposition \ref{prop:tara} with \eqref{VTW} that

\noi
\begin{align}\label{Q1}
d_{\rm Wass}\big( \tfrac{\wt F_\theta - \wt F_\w}{\sqrt{2}}  , Y \big)
\leq |  {\rm Cov}(   \wt F_\theta , \wt F_\w) | + \sum_{j=1}^3 \g_j(\theta, \w),
\end{align}

\noi
where $ \g_j(\theta, \w)$, $j=1,2,3$, are defined as in \eqref{tara2}
with $F =  \tfrac{\wt F_\theta - \wt F_\w}{\sqrt{2}}$
and $p=1+\al$,  $q = \min\{2, 1+2\al\}$ for some $\al\in(0, 1]$; 
namely, ignoring the constants

\noi
\begin{gather*}
  \g_1
 \coloneqq
  \bigg[\int_{\mathbf{Z}}  \bigg( \int_{\mathbf{Z}}  \big\| D_{\xi_2}(\wt F_\theta - \wt F_\w) \big\|_{2p}
   \big\| D_{\xi_1}D_{\xi_2}( \wt F_\theta - \wt F_\w) \big\|_{2p}  \,  m(d\xi_2) \bigg)^{1+\al}
  m(d\xi_1)\bigg]^{\frac{1}{1 + \al}},\\
  \g_2
 \coloneqq
 \bigg[ \int_{\mathbf{Z}}  \bigg( \int_{\mathbf{Z}}   \big\| D_{\xi_1}D_{\xi_2}  (\wt F_\theta - \wt F_\w) \big\|^2_{2p}  \,
   m(d\xi_2) \bigg)^{1+\al}
  m(d\xi_1) \bigg]^{\frac{1}{1 + \al}},
  \shortintertext{and}
  \g_3 \coloneqq
      \int_{\mathbf{Z}} \big\| D_\xi ( \wt F_\theta - \wt F_\w) \big\|_{q+1}^{q+1}  \, m(d\xi).
\end{gather*}


\noi \textbf{Step (i).} The estimate for ${\rm Cov}(\wt F_\theta , \wt F_\w)$ is already there, 
if one notices that ${\rm Cov}(\wt F_\theta , \wt F_\w) = \E [H_{\theta} H_{\w}]$, 
where $H$ is defined as in \eqref{def:h} with $f(x) = x$ for all $x\in \R$. 
As a consequence of  \eqref{com3}, for all $0< \theta < \w$,\footnote{To limit the length of this note, 
here
we take a shortcut by using directly \eqref{com3} although this does not
really fit the spirit of the second proof. Instead of an application of \eqref{com3},
the explicit chaos expansion of $\wt F_\theta$ can be used to verify
the bound \eqref{Q2};
we leave this for interested readers. It is worth pointing out that this approach based on the chaos expansion requires $f (x) \equiv x$ for all $x \in \R$ in \eqref{def:h}, and
the obtention of \eqref{com3} for general $f$ crucially relies on the Clark-Ocone formula
for $f(\wt F_\theta)$,
which is not available in the colored-in-time setting.
}  
\noi
\begin{align}\label{Q2}
|  {\rm Cov}(   \wt F_\theta , \wt F_\w) | \les (\theta/\w)^{\frac12}.
\end{align}


\noi \textbf{Step (ii).} In this step we provide the estimate for $\gamma_1$. It is clear that
\begin{align}
\begin{aligned}
  \g_1^{1 + \al}(\theta, \w) &
  \lesssim
  \mathbf{T}_1(\theta, \theta) + \mathbf{T}_1(\theta, \w)
 + \mathbf{T}_1(\w, \theta)  + \mathbf{T}_1(\w, \w) ,
 \end{aligned}
\label{gamma1a}
\end{align}

\noi
where
\[
 \mathbf{T}_1(\theta, \w) \coloneqq
  \int_{\mathbf{Z}}  \bigg[ \int_{\mathbf{Z}}  \| D_{\xi_2} \wt F_\theta  \|_{2p}
   \| D_{\xi_1}D_{\xi_2} \wt F_\w \|_{2p}  \,  m(d\xi_2) \bigg]^{1+\al}
  m(d\xi_1).
\]
Note that $ \mathbf{T}_1(\theta, \theta)$ and
 $\mathbf{T}_1(\w, \w)$ have been already dealt with in \cite{BZ23}, and we have

 \noi
 \begin{align} \label{gamma1b}
 \mathbf{T}_1(\theta, \theta)  \les \theta^{-\al} \quad \text{and} \quad  \mathbf{T}_1(\w, \w) \les \w^{-\al};
 \end{align}
 see equations (4.21)--(4.24) therein. The other two terms in \eqref{gamma1a} can be addressed similarly.
 Using Proposition \ref{prop:BZ23}, we can write next inequities analogous to \eqref{com2},

 \noi
\begin{align}\label{com2b}
\begin{aligned}
\| D_{r, y, z} F_\theta \|_{2p}
&\les |z| \cdot \int_{-\theta}^\theta G_{t_0}( x-y)dx,  \\
\|  D_{r_1, y_1, z_1}  D_{r_2, y_2, z_2} F_\w \|_{2p}
&\les |z_1z_2|\cdot \int_{-\w}^\w G_{t_0}(x-y_1) G_{t_0}(y_2-y_1)dx,
\end{aligned}
\end{align}

\noi
where $t_0$ is fixed as in Proposition \ref{prop:BZ23};
see also equations (4.18)--(4.20) in \cite{BZ23}\footnote{The bounds
in \eqref{com2b} follow easily from those in  \cite[(4.18)--(4.20)]{BZ23}
and triangle inequality with the fact $G_{t_0 - r} \leq G_{t_0}$. }.

Now, we are ready to estimate $\mathbf{T}_1(\theta, \w)$
by utilizing  \eqref{com2b} and Theorem \ref{thm-cls-clt}-(ii):

\noi
\begin{align*}
\mathbf{T}_1(\theta, \w) 
&\leq \frac{1}{(\s_\theta \s_\w)^{1+\al}}
\int_0^{t_0} \int_{\R\times\R_0} m(d \xi_1) \, |z_1|^{1+\al}
\bigg[ \int_0^{t_0} \int_{\R\times\R_0} m(d \xi_2)\, |z_2|^2 \notag \\
&\qquad
\cdot \bigg( \int_{-\theta}^\theta G_{t_0}(x_2 - y_2) dx_2 \bigg)
 \bigg( \int_{-\w}^\w G_{t_0}(x_1 - y_1)G_{t_0}(y_2 - y_1)    dx_1 \bigg)
 \bigg]^{1+\al}  \notag \\
& \les  \frac{1}{( \theta \w)^{\frac{1+\al}{2}}} 
 \int_{\R} dy_1 \widetilde{\mathbf{T}}_1(\theta, \w, y_1) ^{1+\al} ,
 \end{align*}
 
 \noi
 where $m(d \xi_i) = d r_i d y_i \nu(d z_i)$ for $i = 1,2$ (see Section \ref{sec_pre});
\[
   \widetilde{\mathbf{T}}_1(\theta, \w, y_1) \coloneqq \int_{\R}  dy_2  \int_{-\theta}^\theta   dx_2   \int_{-\w}^\w  dx_1
  G_{t_0}(x_2 - y_2)
 G_{t_0}(x_1 - y_1)G_{t_0}(y_2 - y_1) \leq t_0^3;
\]
  and in the last inequality we used   the finiteness of $m_2$ and $m_{1+\al}$ (see Footnote \ref{note3}). Therefore, an application of Jensen's inequality implies that

\noi
 \begin{align}
 \mathbf{T}_1(\theta, \w) &\les  \frac{t^{3\al}_0}{( \theta \w)^{\frac{1+\al}{2}}}
 \int_{\R} dy_1   \widetilde{\mathbf{T}}_1(\theta, \w, y_1)  
  \les \frac{\theta}{   ( \theta \w)^{\frac{1+\al}{2}}} \leq (\theta/\w)^{ \frac{1+\al}{2}} \label{gamma1c}
\end{align}
for $1\leq \theta \leq \w$, where in the second-to-last inequality, we have simply performed integration
in the exact order of $dx_1, dy_1, dy_2, dx_2$.
In the same way, we can obtain

\noi
\begin{align}
\mathbf{T}_1(\w, \theta)
 \leq (\theta/\w)^{ \frac{1+\al}{2}} \quad\text{for $1\leq \theta \leq \w$}.
\label{gamma1d}
\end{align}

\noi
Therefore, in view of \eqref{gamma1a}, \eqref{gamma1b},
\eqref{gamma1c}, and \eqref{gamma1d},
we get

\noi
\begin{align}\label{est1}
\g_1(\theta, \w)  \les \theta^{-\al} + \w^{-\al} + (\theta/\w)^{\frac{1+\al}2}.
\end{align}

\medskip

 \noi \textbf{ Step (iii).} 
For the term $\gamma_2$, we can write

\noi
\begin{align}
\begin{aligned}
 \g_2^{1 + \al}(\theta, \w)
 &\les   \int_{\mathbf{Z}}  \bigg[ \int_{\mathbf{Z}}   \| D_{\xi_1}D_{\xi_2} \wt F_\theta   \|^2_{2p}  \,
   m(d\xi_2) \bigg]^{1+\al}
  m(d\xi_1)  \\
  &\qquad +   \int_{\mathbf{Z}}  \bigg[ \int_{\mathbf{Z}}   \| D_{\xi_1}D_{\xi_2}  \wt F_\w \|^2_{2p}  \,
   m(d\xi_2) \bigg]^{1+\al}
  m(d\xi_1) \les \theta^{-\al} + \w^{-\al},
 \end{aligned}
   \label{est2}
 \end{align}
which is essentially done in \cite[(4.25)--(4.27)]{BZ23}.

\medskip

 \noi \textbf{Step (iv).} Now we consider the last term $\g_3(\theta, \w)$
 with $q = \min\{ 2,  1 + 2\al\}\in( 1, 2].$

\noi
\begin{align}
\begin{aligned}
 \g_3(\theta, \w)
 \les  \int_{\mathbf{Z}} \| D_\xi  \wt F_\theta   \|_{q+1}^{q+1}  \, m(d\xi)
 + \int_{\mathbf{Z}} \| D_\xi  \wt F_\w \|_{q+1}^{q+1}  \, m(d\xi)
 \les \theta^{-\frac{q-1}{2}} + \w^{-\frac{q-1}{2}},
 \end{aligned}
 \label{est3}
\end{align}

\noi
where the last step
is  essentially done in \cite[(4.28)--(4.30)]{BZ23}.
Claim \eqref{claima} follows from \eqref{Q1}, \eqref{Q2},
\eqref{est1}, \eqref{est2}, and \eqref{est3}.
Hence, the proof is complete.\qed


\begin{remark} \rm \label{rem3}
In this proof, we merged Ibragimov-Lifshits' criteria and the second-order Gaussian Poincar\'e inequalities.
The strategy was largely motivated by earlier works
\cite{BNT, CZthesis, Zheng17, AN22} based on a combination of Ibragimov-Lifshits' method
and the Malliavin-Stein approach (\cite{blue}).
With the chaos expansion, papers
\cite{BNT, CZthesis, Zheng17} have established sufficient conditions
(in terms of contractions) for the ASCLT on Gaussian, Poisson, and Rademacher
settings. Our strategy shares the same root as the aforementioned references
but
 differs by incorporating the second-order Gaussian Poincar\'e inequalities.
This novelty avoids lengthy computation
of asymptotic negligibility of the (star-)contractions, and will be further illustrated in the future.
\end{remark}

\appendix

\section{Proofs of Lemma \ref{lem1} and Proposition \ref{prop:IL}}\label{APP}

 \begin{proof}[Proof of Lemma \ref{lem1}]
 Let us first compute the second moment of $L_T$ in \eqref{com4}:

 \noi
 \begin{align*}
 \E\big[ L_T^2 \big]
 &= \frac{1}{(\log T)^2} \int_1^T \int_1^T \frac{1}{\theta \w}  \E[ H_\theta H_\w ] \,d\theta d\w \\
 &\leq \frac{2}{(\log T)^2} \int_{1 < \theta < \w \leq T}  \frac{1}{\theta \w}
 C_\beta  \big( \theta / \w \big)^{\beta} \,d\theta d\w
 \leq \frac{2C_\beta}{\beta \log T}.
 \end{align*}
 In particular, we get (by Fubini)
 \[
 \E \sum_{k=1}^\infty L^2_{2^{k^2}} < +\infty,
\quad
 \text{and therefore  $L_{2^{k^2}} \xrightarrow{a.s.} 0$
 as $k\to+\infty$.}
 \]

\noi
Next, we will show the almost sure convergence
 along the continuum parameter (as $T\to+\infty$).
By the uniform boundedness of $\{H_{\theta}\}_{\theta > 0}$, one can find a constant $M > 0$ such that
$
 | H_\theta | \leq M$,
  $\forall \theta > 0$,
 with probability one. Then, for any $T > 1$, one can find some (unique) $k = k_T\in\N_{\geq 0}$
 such that
 \begin{align}\label{KT}
 2^{k^2_T} \leq T <  2^{(k_T+1)^2}.
 \end{align}
 It follows that

 \noi
 \begin{align*}
| L_T |
 & =  \Big| \frac{\log (2^{k^2} ) }{\log T} L_{2^{k^2}}
 +  \frac{1}{\log T} \int_{2^{k^2}}^T  \frac{1}{\theta} H_\theta \, d\theta \Big| \\
 &
 \leq  \big| L_{2^{k^2}} \big| + \frac{M}{\log T} \big[  \log T - \log (2^{k^2}   \big) \big]
 \quad\text{with $k = k_T$ as in \eqref{KT}} \\
 &\leq
 \big| L_{2^{k^2}} \big| + \frac{M}{k^2} \big[  (k+1)^2 - k^2  \big],
  \end{align*}

 \noi
 which goes to $0$ as $T\to\infty$ ($k = k_T\to\infty$ as well).
 Hence the proof is complete. \qedhere

 \end{proof}

\begin{proof}[Proof of Proposition \ref{prop:IL}]
The proof consists of two parts: (i) 
for any $s\in\R$,
$\K_t(s)\to 0$ a.s. as $t\to\infty$; and (ii)
for $\PP$-almost
every $\omega\in\Omega$
and for any sequence $t_n\uparrow \infty$ (as $n\to\infty$),
  the family of probability measures
$\{ \nu_{t_n}^\o:  n\geq 1 \}$ is tight,
where we have used the notation $\nu_T^\o$ from \eqref{def2}.
Then, we can deduce from (i) and (ii) together that
$\PP (\{ \o\in\Omega: \nu^\o_T\LRA \g\,\,\text{as $T\to+\infty$} \}) =1$.
That is, ASCLT holds for $\{F_\theta\}_{\theta\geq 1}$.


To show part (i), we fix any $s\in\R$. Let $h>1$ and define
\[
I_j = [ e^{h^j},  e^{h^{j+1}} ], \quad j\in\N_{\geq 0}.
\]
It is clear that $t\mapsto \E  \big[ |\K_t(s)|^2 \big]$ is  continuous, and therefore
there exists some (deterministic) $s_j\in I_j$ such that

\noi
\begin{align}
\begin{aligned}
 \E  \big[ |\K_{s_j}(s)|^2 \big]
 = \min_{t\in I_j}    \E  \big[ |\K_t(s)|^2 \big]
 &\leq  \bigg(\int_{I_j}
    \dfrac{1  }{t  \log t} dt  \bigg)^{-1}
      \int_{I_j}
    \dfrac{\E \big[  |\K_t(s)|^2\big]  }{t  \log t} dt \\
    &= \frac{1}{\log h}   \int_{I_j}
    \dfrac{\E \big[  |\K_t(s)|^2\big]  }{t  \log t} dt,
\end{aligned}
\label{cont}
\end{align}

\noi
which is summable in $j$ by assumption \eqref{IL}.
It follows from Fubini that

\noi
\begin{align} \label{IL1}
\text{almost surely,
$\K_{s_j}(s) \to 0$ as $j\to\infty$.}
\end{align}

\noi
Next, for    $e^{h^j} \leq  a < b \leq e^{h^{j+1}}$, we get

\noi
\begin{align}\label{IL2}
\begin{aligned}
\big| \K_a(s) - \K_b(s) \big|
&= \Big| \Big( \frac{1}{\log a} -\frac{1}{\log b} \Big)\int_1^a \frac{1}{\theta} e^{isF_\theta}\, d\theta
 - \frac{1}{\log b} \int_a^b \frac{1}{\theta} e^{isF_\theta}\, d\theta \Big| \\
 &\leq  2(1- \tfrac{1}{h} ).
\end{aligned}
\end{align}

\noi
Thus,  part (i) follows from \eqref{IL1} and \eqref{IL2}: almost surely,

\noi
\begin{align}\label{IL3}
 \limsup_{t\to\infty} | \K_t(s) | \leq 2(1- \tfrac{1}{h} ) \quad \text{for all } h > 1.
\end{align}


Now, let us continue with part (ii).
First, we deduce from the same continuity argument as in
\eqref{cont} that for $r > 0$, there exists some (deterministic)
$s_j^\ast \in I_j$ (depending on $r$ as well) such that

\noi
\begin{align*}
\E\int_{-r}^r  | \K_{s_j^\ast}(s) |^2 ds
&= \min_{t\in I_j} \E\int_{-r}^r  | \K_{t}(s) |^2 ds
\leq \frac{1}{\log h}  \int_{I_j} \frac{dt}{t \log t}  \bigg( \E\int_{-r}^r  | \K_{t}(s) |^2 ds \bigg),
\end{align*}

\noi
which, together with Fubini and the assumption \eqref{IL}, implies that
\begin{align*}
\E \sum_{j}  \int_{-r}^r  | \K_{s_j^\ast}(s) |^2 ds
\leq \frac{2r}{\log h} \sup_{|s|\leq r}
 \int_{2}^\infty \frac{1}{t \log t} \E \big[   | \K_{t}(s) |^2 \big]  dt <+\infty.
\end{align*}

\noi
Thus, it holds almost surely that
$\int_{-r}^r  | \K_{s_j^\ast}(s) |^2 ds \to 0$ as $j\to+\infty$.
Therefore, we can obtain by the same arguments as in
\eqref{IL1}, \eqref{IL2}, and \eqref{IL3} that
with probability one,
$
\int_{-r}^r  | \K_{t}(s) |^2 ds \to 0
$
as $t\to+\infty$. Then, we get from H\"older inequality that
\noi
\begin{align}\label{IL3b}
\text{almost surely}, \quad
\int_{-r}^r  \phi_{\nu_t^\o}(s)  ds \to \int_{-r}^r  \phi_{\g}(s)  ds
\quad\text{as $t\to+\infty$,}
\end{align}

\noi
where $\phi_\mu(s)=\int_{\R} e^{is x}\mu(dx)$ denotes the characteristic
function of a probability measure $\mu$ on $\R$, and $\g$ denotes
the standard Gaussian measure on $\R$ in this paper.
Using the fact $|\sin(z)| \leq |z|$, $\forall z\in\R$, we deduce that

\noi
\begin{align} \label{IL4}
\begin{aligned}
\mu\big(\{x\in\R: |x| \leq \tfrac{1}{r^2} \}\big)
&\geq \int_{|x|\leq r^{-2}} \frac{\sin(rx)}{rx} \mu(dx)
\\
&\geq  \int_{\R} \frac{\sin(rx)}{rx} \mu(dx)  - r
= \frac{1}{2r} \int_{-r}^r \phi_\mu(s)ds  -r.
\end{aligned}
\end{align}

\noi
It follows from \eqref{IL4} and \eqref{IL3b} that
for $\PP$-almost every $\omega\in\Omega$ and for
every $\eps > 0$,

\noi
\begin{align*}
\liminf_{t\to+\infty} \nu_t^\o(|x| \leq r^{-2})
& \geq  \frac{1}{2r} \int_{-r}^r \phi_\g(s)ds  -r \geq 1 - \frac{\eps}{2}
\end{align*}
for small enough $r = r_\eps$ that does not depend on $\o$.
For any increasing divergent sequence $\{t_n\}$, the above bound indicates that
there is some $N= N_\eps$ such that
\begin{align} \label{IL5}
 \nu_{t_n}^\o(|x| \leq r^{-2}_\eps) \geq 1 - \eps, \quad\forall n\geq N_\eps;
\end{align}

\noi
while by choosing another small enough $r'_\eps > 0$, we get
\begin{align} \label{IL6}
 \nu_{t_n}^\o(|x| \leq  \tfrac{1}{( r'_\eps)^2}) \geq 1 - \eps, \quad\forall n < N_\eps.
\end{align}
Combining \eqref{IL5} and \eqref{IL6} yields the tightness of $\{ \nu_{t_n}^\o: n\geq 1\}$.
Therefore, the result in part (i) implies that $ \nu_{t_n}^\o\LRA \g$ as $n\to\infty$,
and such weak convergence holds for true for any increasing divergent sequence $\{t_n\}_{n\in \N}$.
Hence, we proved  $\PP \{ \o\in\Omega: \nu^\o_T\LRA \g\,\,\text{as $T\to+\infty$} \}=1$.
\qedhere

\end{proof}

\bibliographystyle{amsplain}

\end{document}